\documentclass[11pt]{amsart}
\usepackage[margin=3cm]{geometry}               
\usepackage[usenames,dvipsnames,svgnames,table]{xcolor}    
\usepackage[all]{xypic}
\usepackage{graphicx}
\usepackage{tikz}
\usetikzlibrary{decorations.pathreplacing}
\usepackage{lineno, hyperref}
\usepackage{amssymb,amsfonts, amsmath, amsthm}
\usepackage{epstopdf,enumitem}
\usepackage{ulem}
\usepackage{cancel}
\usepackage{thmtools} 
\usepackage{mathrsfs}
\usepackage[english]{babel}
\usepackage{csquotes}

\newtheorem*{theorem*}{Theorem}
\newtheorem*{corollary*}{Corollary}

\declaretheorem[style=definition,numberwithin=section]{definition}
\declaretheorem[style=definition,sibling=definition]{remark}

\declaretheorem[style=definition, sibling = definition]{lemma-definition}

\declaretheorem[style=theorem, sibling = definition]{theorem}
\declaretheorem[style=theorem, sibling = definition]{lemma}

\declaretheorem[style=theorem, sibling = definition]{corollary}

\newcommand{\R}{\mathbb R}
\newcommand{\Q}{\mathbb Q}
\newcommand{\N}{\mathbb N}

\newcommand{\res}{\mathrm{res}}

\newcommand{\cU}{\mathcal{U}}
\newcommand{\Div}{\mathrm{div}}
\newcommand{\ov}{\mathrm{ov}}

\newcommand{\alg}{\mathrm{alg}}
\newcommand{\llangle}{\langle\!\langle}
\newcommand{\rrangle}{\rangle\!\rangle}
\newcommand{\mL}{\mathcal L}
\newcommand{\cA}{\mathcal{A}}
\newcommand{\cO}{\mathcal{O}}
\newcommand{\Th}{\mathrm{Th}}
\newcommand{\RCVF}{\mathrm{RCVF}}

\title{Real closed valued fields with analytic structure}

\author[P. Cubides Kovascics]{Pablo Cubides Kovacsics$^{\ast}$}
\address{\hskip-\parindent
Pablo Cubides Kovacsics \\ Laboratoire de math\'ematiques Nicolas Oresme\\ Universit\'e de Caen\\CNRS UMR 6139 
Universit\'e de Caen BP 5186\\
14032 Caen cedex, France. }
\email {pablo.cubides@unicaen.fr}
\thanks{${}^{\ast}$ Supported by the ERC project TOSSIBERG (Grant Agreement 637027) and partially supported by ERC project MOTMELSUM (Grant agreement 615722).}

\author[D. Haskell]{Deirdre Haskell$^{\dag}$}
\address{\hskip-\parindent
Deirdre Haskell \\
Department of Mathematics and Statistics \\
McMaster University \\
1280 Main St W \\
Hamilton ON L8S 4K1 \\
CANADA} 
\email {haskell@math.mcmaster.ca}
\thanks{${}^{\dag}$ Partially supported by NSERC}

\begin{document}


\maketitle

\begin{abstract} 
We show quantifier elimination theorems for real closed valued fields with separated analytic structure and overconvergent analytic structure in their natural one-sorted languages and deduce that such structures are weakly o-minimal. We also provide a short proof that algebraically closed valued fields with separated analytic structure (in any rank) are $C$-minimal. 

\end{abstract}

%
\normalem

\section{Introduction}

Since the pioneering work of Denef and van den Dries in \cite{DD}, subanalytic sets and fields with analytic structure have been intensively studied by various authors (see \cite{Mac-Mar-Dries1, Mac-Mar-Dries2, Mac-Mar-Dries3,CLR, Lip-Rob-2000, schoutens, L, clu-2003, clr-06, vdDHM}). The most complete account to date is the one given by Cluckers and Lipshitz in \cite{CLip} where, inspired by almost all previous work, they provided an abstract setting which isolates the properties a ring of functions should satisfy in order to behave like a ring of analytic functions. When the ambient field is sufficiently tame, for example algebraically closed, real closed or simply henselian, the algebraic properties imposed on the considered rings of functions provide powerful structure theorems taking the form of a quantifier elimination theorem or of a tameness property such as o-minimality of the underlying structure.

In this article, we continue this line of investigation by studying real closed valued fields with analytic structure in a setting which has not been explored before. In previous work on real closed fields with analytic structure, the analytic structure is in some way intrinsic to the ordered structure. One main thrust of research, which is in the same spirit as this investigation, is that of the structures studied by Lipshitz and Robinson in \cite{Lip-Rob-2000} and later generalized by Cluckers, Lipshitz and Robinson in \cite{CLR} and Cluckers and Lipshitz in \cite{CLip}. In this setting, real closed valued fields are studied in a language without the valuation, and the analytic functions under consideration, although potentially defined using the valuation, are restricted to closed boxes $[-1,1]^n$. In particular these structures are o-minimal. The other approach is the general framework of $T$-convex structures introduced by Lewenberg and van den Dries in \cite{lewen-vdD}. In this setting, one studies non-standard models of o-minimal expansions of real closed fields by adding the natural valuation. Non-standard models of the theory of the field of real numbers with restricted analytic functions constitute one of the main examples of this setting. In the presence of the valuation, $T$-convex structures are not o-minimal. However, they are \emph{weakly o-minimal}: every definable subset of the line is a finite union of convex sets. 

In the present article we study real closed valued fields with analytic structure in a language which contains the valuation and in which the underlying rings of analytic functions are restricted to powers of the valuation ring. As the valuation is in the language, the structures we consider are obviously not o-minimal, so we are not in the framework of \cite{Lip-Rob-2000, CLR, CLip}. Furthermore, the valuation can be recovered from each analytic function, which implies our structures are not $T$-convex, since the reduct without the valuation is still not o-minimal. The main results of this paper include a quantifier elimination theorem (Theorem \ref{thm:QE_an}) and a tameness theorem (Theorem \ref{thm:weak-o-min_an}) which can be informally gathered as follows: 

\begin{theorem*} Let $K$ be a real closed valued field with separated analytic $\cA$-structure. Then the theory of $K$ in the one-sorted language of ordered valued fields with separated analytic $\cA$-structure has quantifier elimination and is weakly o-minimal.  
\end{theorem*}

As a corollary (Corollary \ref{cor:overwo}), we obtain that real closed valued fields of rank 1 such as the field of Puiseux series $\bigcup_{n>0}\R(\!(t^{1/n})\!)$ or Hahn fields such as $\R(\!(t^{\Q})\!)$ and $\R(\!(t^\R)\!)$ with (full) overconvergent analytic structure are weakly o-minimal. We also show quantifier elimination for the theory of such structures (see Theorem \ref{thm:QE_ov}). 

\

It is worth noticing that, since real closed valued fields are henselian, the structures considered here fall into the formalism  for henselian fields with analytic structure presented in \cite{CLip}. One of the results of that paper is a quantifier elimination theorem (\cite[Theorem~4.5.15]{CLip}) in a multi-sorted language from which the model-theoretic tameness property called $b$-minimality is deduced (\cite[Theorem~6.3.7]{CLip}). Our original expectation was that $b$-minimality would imply weak o-minimality in the real closed case. However, it is not at all clear from $b$-minimality what the definable sets on the value group and residue field will look like, nor how these descriptions interact with the ordering. By working in the one-sorted language as we do here, we get further information about the definable sets that might not be extracted from the previous results in the multi-sorted language.  

A related point to consider is what happens for algebraically closed valued fields with separated analytic structure. In this case, a quantifier elimination is also available in the one-sorted language (see \cite[Theorem 4.5.15]{CLip}). It is worth observing, as we do in Theorem \ref{thm:Cmin}, that $C$-minimality (the natural tameness notion in this case) follows by a fairly short argument. We do not know if $C$-minimality could be deduced directly from the multi-sorted setting. 

\

The article is organized as follows. In Section \ref{sec:setting}, notation is set and some definitions from \cite{CLip} are recalled. Quantifier elimination is proved in Section \ref{sec:QE}. Section \ref{sec:weak} is divided in two parts: in Section \ref{sec:C-min} we prove $C$-minimality in the algebraically closed case and in Section \ref{sec:weak-o} weak o-minimality is shown in the real closed case. In Section \ref{sec:over} we prove the quantifier elimination theorem for real closed valued fields with overconvergent analytic structure.

\section{The setting}\label{sec:setting}

Let $(K,v)$ be a valued field. Somewhat unconventionally, we write $v$ multiplicatively, that is, we let $(\Gamma_K,\cdot)$ denote a divisible ordered abelian group and $v\colon K\to \Gamma\cup\{0\}$ denote a surjective map satisfying  
\begin{enumerate}
\item $v(x)=0$ if and only if $x=0$;
\item $v(xy)=v(x)v(y)$;
\item $v(x+y)\leqslant \max\{v(x),v(y)\}$. 
\end{enumerate}

We let $K^\circ$ denote the valuation ring of $(K,v)$, $K^{\circ\circ}$ its maximal ideal, $\widetilde{K}:=K^\circ/K^{\circ\circ}$ its residue field and $\res\colon K^\circ\to \widetilde{K}$ the residue map. Given $a\in K$ and $\gamma\in \Gamma_K$, the open ball centered at $a$ of radius $\gamma$ is the set $\{x\in K\mid v(a-x)<\gamma\}$. The closed ball centered at $a$ of radius $\gamma$ is the set $\{x\in K\mid v(a-x)\leqslant \gamma\}$. We treat points of $K$ as closed balls of radius $0$. The set of open balls forms a basis for the topology on $K$. 
The one-sorted language of valued fields $\mathcal{L}_{\Div}$ is the language $(+,-,\cdot, 0,1, \Div)$, where $\Div$ is a binary relation symbol interpreted in a valued field $(K,v)$ by $\Div(x,y) \Leftrightarrow  v(x)\leqslant v(y)$. 

For much, but not all, of the paper, $K$ will be an \emph{ordered, valued} field; that is, the valuation is convex with respect to the ordering. We use the standard absolute value notation $|x|$ to denote the function
\begin{equation*}
|x|:=
\begin{cases} 
x & \text{if } x\geqslant 0\\
-x & \text{if } x<0.  
\end{cases}
\end{equation*}
The language of ordered valued fields is $\mathcal{L}_{\Div}^{\leqslant}:=\mathcal{L}_{\Div}\cup\{\leqslant\}$, where $\leqslant$ is interpreted in an ordered valued field by the ordering.

\subsection{Separated analytic structure}\label{subsec:separated}

The notion of \emph{separated analytic $\mathcal{A}$-structure} on a valued field $K$ is defined in \cite{CLip}. Briefly, it consists of a ring $A$, a collection $\mathcal{A} = (A_{m,n})_{m,n\in\N}$ of subrings of the ring of formal power series over $A$ in the separated variables $x=(x_1,\ldots, x_m)$ and $\rho=(\rho_1,\ldots,\rho_n)$, and a collection $(\sigma_{m,n})$ of homomorphisms from $A_{m,n}$ to the ring of $K^\circ$-valued functions on $(K^\circ)^m\times(K^{\circ\circ})^n$ satisfying natural compatibility conditions. The rings extend polynomial rings over $A$ and, crucially, satisfy Weierstrass division and a strong Noetherian property (see \cite[Sections 4.1 and 4.2]{CLip}). The most obvious example of separated analytic $\mathcal{A}$-structure occurs when $K$ is a complete valued field with rank one value group, and $A=K^\circ$. The rings $A_{m,n}$ consist of the convergent power series in the $x$-variables, and any power series in the $\rho$-variables. These will converge as functions on $(K^\circ)^m\times(K^{\circ\circ})^n$, so the $\sigma_{m,n}$ can be taken to be the homomorphisms which take the power series as an element of the ring to the function defined by evaluating the power series. There are many other examples, as described in \cite[Section~4.4]{CLip}. In general, we may suppose that $\ker(\sigma_{0,0}) =(0)$ and think of $A$ as a subring of $K^\circ$ and $A_{m,n}$ as a subring of $K^\circ[\![x,\rho]\!]$.

The language of valued fields with (separated) analytic $\mathcal{A}$-structure, denoted $\mathcal{L}_\mathcal{A}$, is the language $\mathcal{L}_{\Div}$ together with a unary function symbol $^{-1}$ and an $(m+n)$-ary function symbol for each element $f\in A_{m,n}$ with $m,n\in \N$. The structure $(K,\mathcal{L}_\mathcal{A})$ is defined by interpreting $\mathcal{L}_{\Div}$ in the natural way, $^{-1}$ as the multiplicative inverse function extended to $K$ by setting $0^{-1}=0$, and each $f\in A_{m,n}$ as the function 
\[
f^{\mathcal{L}_\mathcal{A}}(x,\rho):=
\begin{cases}
\sigma_{m,n}(f)(x,\rho), & \text{ if $(x,\rho)\in (K^\circ)^m\times (K^{\circ\circ})^n$}; \\ 
0, & \text{otherwise.}
\end{cases}
\]
We define the language $\mL_{\mathcal{A}}^{\leqslant}$ of \emph{ordered valued fields with analytic $\mathcal{A}$-structure} as the extension of $\mL_{\mathcal{A}}$ by a new binary relation symbol $\{\leqslant\}$. 

\
As always, Weierstrass preparation plays a central role in analysing the definable sets; to use it requires an appropriate notion of regularity. This is one place where we see the different roles of the two kinds of variables. 
The following terminology comes from \cite{L}. Let $x=(x_1,\ldots,x_m)$, $\rho=(\rho_1,\ldots,\rho_n)$, $x'=(x_{m+1},\ldots,x_{m'})$ and $\rho'=(\rho_{n+1},\ldots,\rho_{n'})$. Let $f(x, \rho, x', \rho')$ be an $\mL_\cA$-term in which $^{-1}$ is only applied to terms not involving $x'$ or $\rho'$. Then $f = \sum_{\mu\nu} a_{\mu\nu}(x, \lambda) (x')^\mu (\rho')^\nu$ where the $a_{\mu\nu}$ are $\mL_{\cA}$-terms. The function $f$ is called \emph{preregular of degree $(\mu_0, \nu_0)$} if $a_{\mu_0\nu_0}= 1$ and for $\nu < \nu_0$ and also for $\nu = \nu_0$, $\mu>\mu_0$, $a_{\mu\nu}(x,\lambda)\in K^{\circ\circ}$ for all $x,\lambda$ in $(K^{\circ})^{m}\times (K^{\circ\circ})^{n}$. We
call $f$ \emph{regular of degree $\mu_0$} if it is preregular of degree $(\mu_0,0)$ and \emph{regular of degree $\nu_0$} if it is preregular
of degree $(0, \nu_0)$. 

%
%
%
%
%
%

\section{Quantifier elimination for real closed valued fields with separated analytic structure}\label{sec:QE}

Let $K$ be a real closed valued field with separated analytic $\mathcal{A}$-structure. To show that $Th(K,\mathcal{L}_\mathcal{A}^{\leqslant})$ has quantifier elimination, we use the following result, due to Cherlin and Dickmann, about its associated algebraic theory $\Th(K, \mL_\Div^\leqslant)$. This latter theory corresponds to the $\mathcal{L}_{\Div}^\leqslant$-theory of real closed non-trivially valued fields, commonly denoted by $\RCVF$. 

\begin{theorem}[{\cite[Section 2]{CD}}]\label{thm:QE} The theory $\RCVF$ has quantifier elimination. 
\end{theorem}

\begin{theorem}\label{thm:QE_an} The theory $Th(K,\mathcal{L}_\mathcal{A}^{\leqslant})$ has quantifier elimination. 
\end{theorem}

\begin{proof} The proof follows the strategy of \cite{L} (which is also based on \cite{DD}) in which the key step is the following so-called \emph{Basic Lemma}. The method to derive the theorem from the Basic lemma is exactly as in \cite[3.8.5]{L}. The proof of the Basic lemma is given below. 
\end{proof}

\begin{lemma}[Basic Lemma] Let $\varphi(x,\lambda,y,\rho)$ be an $\mathcal{L}_{\mathcal{A}}^{\leqslant}$-quantifier free formula in which $^{-1}$ is only applied to terms not involving $y$ or $\rho$. Then there is an $\mathcal{L}_\mathcal{A}^{\leqslant}$-quantifier free formula $\psi(x,\lambda)$ equivalent to $\exists y\exists \rho\, \varphi(x,\lambda,y,\rho)$. 
\end{lemma} 

\begin{proof} Let $\varphi(x,\lambda,y,\rho)$ be as in the statement of the lemma with $y=(y_1,\ldots,y_M)$ and $\rho=(\rho_1,\ldots,\rho_N)$. By standard reductions we may assume that $\exists y\exists \rho\varphi$ is of the form 
\[
\varphi_0(x,\lambda)\wedge \exists y\exists \rho\left(\bigwedge_{i=1}^4\varphi_i(x,\lambda,y,\rho)\right), 
\]
where 
\begin{align*}
\varphi_1(x,\lambda,y,\rho)\  \text{is the formula}  & \bigwedge_{i\in I_1} (f_i(x,\lambda,y,\rho)=0); \\ 
\label{eq:2} \varphi_2(x,\lambda,y,\rho)\  \text{is the formula} & \bigwedge_{i\in I_2} (0< f_i(x,\lambda,y,\rho));\\ 
\varphi_3(x,\lambda,y,\rho)\ \  \text{is the formula} & \bigwedge_{i\in I_3} (v(f_i(x,\lambda,y,\rho))=1); \\ 
\varphi_4(x,\lambda,y,\rho) \text{is the formula} & \bigwedge_{i\in I_4, j\in I_5} v(f_i(x,\lambda,y,\rho))<v(f_j(x,\lambda,y,\rho)) , 
\end{align*}
where the $x$ and $y$ variables range over the valuation ring, the $\lambda$ and $\rho$-variables range over the maximal ideal and each $f_i$ is of the form $f_i = a_i(x,\lambda)g_i(x,\lambda,y,\rho)$ with $a_i(x,\lambda)$ an $\mathcal{L}_{\mathcal{A}}$-term and $g_i\in A_{M,N}$ preregular of degree $(\mu_{i,0},\nu_{i,0})$. We proceed by induction on $(M,N)$. 

\noindent
\textbf{Case 1:} Suppose $N=0$. This is handled word for word as in \cite[3.13 Case 1]{L}. 

\noindent
\textbf{Case 2:} Suppose $N>0$ and let $I:=\bigcup_{s=1}^5 I_s$. Since $g_i$ is preregular of degree $(\mu_{i,0},\nu_{i,0})$, writing $g_i=\sum_{\nu} b_{i,\nu}(x,\lambda,y)\rho^\nu$ we have that $b_{i,\nu_{i,0}}$ is preregular of degree $\mu_{i,0}$. As in \cite[3.13 Case 2]{L}, we further split into cases with respect to the following disjunction:
\[
\varphi\longleftrightarrow \left([\bigvee_{i\in I} \varphi\wedge v(b_{i,\nu_{i,0}})<1)] \vee (\varphi\wedge \bigwedge_{i\in I} v(b_{i\nu_{i,0}})=1) \right).
\]

\noindent
\textbf{Case 2a:} Fix $j\in I$ and suppose our formula is of the form $\varphi\wedge v(b_{j,\nu_{j,0}})<1$. We proceed as in \cite[3.13 Case 2a]{L}: one directly shows in this case that for all $i\in I$, under the assumption $v(b_{j,\nu_{j,0}})<1$, $f_i=r_i$ for an element $r_i\in A_{M,N}$ which is polynomial in $y_M$. By uniformly replacing $f_i$ by $r_i$, we obtain the result by Theorem \ref{thm:QE} and the induction hypothesis. 

\noindent
\textbf{Case 2b:} Suppose our formula is $\varphi\wedge \bigwedge_{i\in I} v(b_{i,\nu_{i,0}})=1$. Here we cannot proceed as in \cite[3.13 Case 2b]{L} since we need to take into account potential sign changes. Let $\mathcal{P}(I_2)$ denote the set of subsets of $I_2$. For $B\in \mathcal{P}(I_2)$ let $\theta_B$ be the formula 
\[
\varphi\wedge \bigwedge_{i\in I} v(b_{i,\nu_{i,0}})=1\wedge \bigwedge_{i\in B} 0<b_{i,\nu_{i,0}} \wedge \bigwedge_{i\in I_2\setminus B} b_{i,\nu_{i,0}}<0. 
\]
We further split into cases with respect to the disjunction $\bigvee_{B\in \mathcal{P}(I_2)} \theta_B$ (which is equivalent to $\varphi\wedge \bigwedge_{i\in I} v(b_{i,\nu_{i,0}})=1$). Fix $B\in \mathcal{P}(I_2)$. Let $y_{M+1}$ be a new variable and consider the following formula which is equivalent to $\theta_B$
\[
\exists y_{M+1} [v(y_{M+1})= 1 \wedge \theta_B \wedge y_{M+1}\prod_{i\in I} b_{i,\nu_{i,0}} -1 = 0].
\]
Let $f_i':=a_i(x,\lambda)g_i'(x,\lambda,y,\rho)$ where 
\[
g_i'=(\sum_{\nu\neq\nu_{i,0}} b_\nu\rho^\nu) y_{M+1}\prod_{j\neq i} b_{j,\nu_{j,0}} + \rho^{\nu_{i,0}}. 
\]
We will replace each $f_i$ by $f_i'$, but to get an equivalent formula, we need to take care of the potential change of sign of $f_i'$ for $i\in I_2$. Since the sign of $f_i'$ only depends on the sign of $b_{i,\nu_{i,0}}$ (which is specified by $B$), we replace the corresponding order relation in $\varphi_2$ accordingly, that is, we replace $\varphi_2$ inside $\theta_B$ by the conjunction 
\[
\bigwedge_{i\in B} (0< f_i'(x,\lambda,y,\rho)) \wedge \bigwedge_{i\in I_2\setminus B} (f_i'(x,\lambda,y,\rho)<0).
\]
Call this new equivalent formula $\theta_B'$. Note that $f_i'$ is preregular of degree $(0,\nu_{i,0})$ and $Y_{M+1}\prod_{i\in L} b_{i,\nu_{i,0}} -1$ is preregular in $(y,y_{M+1})$ of degree  $(\sum \mu_{i,0}, 0)$. These conditions allow us to conclude the proof following the same argument as in \cite[3.13 Case 2b]{L}. 
\end{proof}

\section{Further tameness properties of fields with analytic structure}\label{sec:weak}

In this section we show that $Th(K,\mathcal{L}_{\mathcal{A}}^{\leqslant})$ is weakly o-minimal when $K$ is a real closed valued field, and is $C$-minimal when $K$ is an algebraically closed valued field. We recall the definitions.

Let $\mL$ be a language extending $\{\leqslant\}$. A totally ordered structure $(K,\mL)$ is said to be \emph{weakly o-minimal} if every definable subset $X\subseteq K$ is a finite union of points and convex sets. The theory $\Th(K,\mL)$ is \emph{weakly o-minimal} if every elementarily equivalent structure $(K',\mL)\equiv (K, \mL)$ is weakly o-minimal. Analogously, let $\mL$ be a language extending the language $\mL_{\Div}$. An expansion of a valued field $(K, \mL)$ is called \emph{$C$-minimal} if for every elementarily equivalent structure $(K',\mL)\equiv (K,\mL)$, every definable subset $X\subseteq K$ is a finite boolean combination of balls (either closed or open).

In \cite{lip-rob-98}, Lipshitz and Robinson showed that if $K$ is an algebraically closed valued field of rank 1 with separated analytic $\mathcal{A}$-structure, then $(K,\mL_\mathcal{A})$ is $C$-minimal. Here we give a shorter proof of the more general statement that applies when $K$ is a field of any rank. A very similar argument gives the weak o-minimality result when $K$ is a real closed valued field of any rank. In both cases, we need to understand the definable functions in one variable on an arbitrary model $L$ of the theory. There is a by-now standard method to incorporate the analytic structure on $L$ into parameters of the power series from $\mathcal{A}$, creating a new separated analytic structure $\mathcal{A}(L)$. This method is explained in detail in \cite[Section 4.5]{CLip}. Furthermore, the separated analytic $\mathcal{A}$-structure on $K$ extends uniquely to a separated analytic $\mathcal{A}$-structure on $K_\alg$, the algebraic closure of $K$ (see \cite[Theorem 4.5.11]{CLip}). This then gives rise to a precise description of the definable functions in one variable. We use the following theorem.

\begin{theorem}[{\cite[Theorem 5.5.3]{CLip}}]\label{thm:terms} Let $x$ be a single variable and let $t(x)$ be a term of $\mathcal{L}_{\cA}$. 
There is a finite set $S\subseteq K_{\alg}^\circ$ and a finite cover of $K_{\alg}^\circ$ by $K$-annuli $\mathcal{U}_i$ such that for each $i$ there is a rational function $R_i \in K(x)$ and a strong unit $E_i\in \mathcal{O}^\dagger(\mathcal{U}_i)$ such that 
\[
t(x)\vert_{\mathcal{U}_i \setminus S}=R_i(x) \cdot E_i^\sigma \vert_{\mathcal{U}_i \setminus S}(x).  
\]
\end{theorem}

Definitions of $K$-annuli and the $K$-algebras $\mathcal{O}^\dagger(\mathcal{U}_i)$ and $\mathcal{O}^\sigma(\mathcal{U}_i)$ can be found in \cite[Section 5.1]{CLip}. There are two key points for us when applying this theorem. The first is that the $K$-annuli are quantifier-free definable in $\mL_{\Div}$. The second is the definition of strong unit, which we repeat below. 

\begin{definition} A unit $f^\sigma$ in $\cO_K^{\sigma}(\varphi)$ is called a \emph{strong unit} if 
\begin{enumerate}
\item there is some $\ell\in \N$ and $c\in K^\times$ such that $v(f^\sigma(x)^\ell)=v(c)$ for all $x\in \cU_\varphi$ and   
\item there exists a non-zero polynomial $P(\xi)\in \widetilde{K}[\xi]$ such that $P(\res(c^{-1}f^\sigma(x)^\ell))=0$ for all $x\in \cU_\varphi$. 
\end{enumerate}
\end{definition}
When the value group is divisible we can always take $\ell=1$ in the previous definition.

\subsection{$C$-minimality of algebraically closed fields with separated analytic structure}\label{sec:C-min}

The following result is morally attributed to Robinson. 

\begin{theorem}[Robinson]\label{fact:Cmin} The theory of algebraically closed valued fields has quantifier elimination in $\mL_{\Div}$. In particular, algebraically closed valued fields are $C$-minimal. 
\end{theorem}

\begin{corollary}\label{coro:Cmin} Algebraically closed valued fields are $C$-minimal. 
\end{corollary}

\begin{theorem}[{\cite[Theorem 4.5.15]{CLip}}]\label{thm:QE_ACVF_an} Let $K$ be an algebraically closed valued field with separated analytic $\cA$-structure. Then $K$ admits elimination of quantifiers in $\mL_{\cA}$.
\end{theorem}

 \begin{theorem}\label{thm:Cmin} Let $K$ be an algebraically closed valued field with analytic $\cA$-structure. Then, $(K,\mL_\cA)$ is $C$-minimal.  
\end{theorem}

\begin{proof}
It suffices to show that in every elementary extension of $(K,\mathcal{L}_\mathcal{A})$ definable sets in one variable are finite boolean combinations of balls. Let $L$ be any such extension and $X\subseteq L$ be a definable set. Let $x$ be a variable, $\varphi(x,y)$ be an $\mathcal{L}_{\mathcal{A}}$-formula and $a\in L^{|y|}$ be such that $X$ is defined by $\varphi(x,a)$. Without loss of generality we may suppose that $x$ varies over $L^\circ$ and that $a\in (L^\circ)^{|y|}$. Each term $t(x,a)$ in $\varphi$ defines the same function as an $\mathcal{L}_{\mathcal{A}(L)}$-term $t_a'(x)$. Let $\varphi'(x)$ be the corresponding $\mathcal{L}_{\mathcal{A}(L)}$-formula where every term $t(x,a)$ in $\varphi(x,a)$ is replaced by $t_a'(x)$. Note that $\varphi$ and $\varphi'$ define $X$. By Theorem \ref{thm:QE_ACVF_an} applied to $\mathrm{Th}(L,\mathcal{L}_{\mathcal{A}(L)})$, we may assume that $\varphi'(x)$ is an atomic $\mathcal{L}_{\mathcal{A}(L)}$-formula. We are reduced then to the following three cases for $t(x),s(x)$ two $\mathcal{L}_{\mathcal{A}(L)}$-terms
\begin{enumerate}
\item $\varphi'$ is of the form $t(x)=0$;
\item $\varphi'$ is of the form $v(t(x))\leqslant v(s(x))$. 
\end{enumerate} 

By Theorem \ref{thm:terms}, we may further suppose that, after possibly partitioning $L^\circ$ into a definable finite partition of $L$-annuli $\cU_i$, there is a finite set $S\subseteq L^\circ$ such that on each annulus $\cU_i$ we have  
\[
t(x)\vert_{\mathcal{U}_i \setminus S}=(R_i(x) \cdot E_i^\sigma) \vert_{\mathcal{U}_i \setminus S}(x) \text{ and }  s(x)\vert_{\mathcal{U}_i \setminus S}=(R_i'(x) \cdot F_i^\sigma) \vert_{\mathcal{U}_i \setminus S}(x), 
 \]
where $R_i, R_i'$ are rational functions over $L$ and $E_i^{\sigma}, F_i^{\sigma}\in \mathcal{O}_L^{\sigma}(\mathcal{U}_i)$ are strong units. Note that $L$-annuli are finite boolean combinations of balls. We split in cases: 

$(1)$ Suppose that $\varphi'$ is of the form $t(x)=0$. Let $c\in L$ be such that $v(E_i^{\sigma})=v(c)$ for all $x\in \cU_i$. Then, the formula $\varphi'$ restricted to $X\setminus S$, is equivalent to $R_i(x)=0$, which defines a finite set of points. 

$(2)$ Suppose that $\varphi'$ is of the form $v(t(x))\leqslant v(s(x))$. Since both $E_i^\sigma$ and $F_i^\sigma$ are strong units, there are $b,c\in L$ such that for all $x\in \cU_i$, $v(E_i^\sigma(x))=v(c)$ and $v(F_i^\sigma(x))=v(b)$. Then, we have that $\varphi$, restricted to $X\setminus S$, is equivalent to $v(c)v(R_i(x))\leqslant v(b)v(R_i'(x))$, and the result follows from Corollary \ref{coro:Cmin}. 
\end{proof}

\subsection{Weak o-minimality}\label{sec:weak-o}

Let $K$ be a real closed valued field with separated $\cA$-analytic structure. The proof that the theory $\Th(K,\mathcal{L}_{\cA}^{\leqslant})$ is weakly o-minimal follows the same strategy as the proof of Theorem \ref{thm:Cmin}. In particular, we use the following corollary which is an easy consequence of Theorem \ref{thm:QE}. 

\begin{corollary}\label{cor:RCVF_weak} Every real closed valued field is weakly o-minimal ($\RCVF$ is weakly o-minimal).  
\end{corollary}

Because of the ordering of $K$, we need some further lemmas. 

\begin{lemma}\label{lem:traces} Let $K$ be a real closed valued field with separated $\cA$-analytic structure. Let $X\subseteq K_\alg$ be an $\mL_{\cA}$-definable set. Then $X\cap K$ is $\mL_{\Div}$-definable. In particular, if $\cU_\varphi$ is a $K$-annulus, $\cU_\varphi(K)$ (its trace on $K$) is $\mL_{\Div}$-definable. 
\end{lemma} 

\begin{proof} By Theorem \ref{thm:Cmin}, $X$ is a finite boolean combination of balls, so we may suppose that $X$ is a ball. If it is a single point, the result is trivial, so we may suppose that 
$X=\{x\in K_\alg\mid v(x-a)\square \gamma \}$ where $\gamma\in v(K_\alg)=v(K)$ and $\square$ is either $\leqslant$ or $<$. Suppose that $X\cap K\neq \emptyset$ and let $b\in X\cap K$. Then $X=\{x\in K_\alg\mid v(x-b)\square \gamma \}$, and thus $B\cap K=\{x\in K\mid v(x-b)\square \gamma \}$. The last statement follows since $K$-annuli are $\mL_{\Div}$-definable. 
\end{proof}

\begin{lemma}\label{lem:sign_change} Let $\varphi$ be a $K$-annulus formula and $f^\sigma$ be a strong unit in $\cO_K^{\sigma}(\varphi)$. Then there is a finite partition of $\cU_\varphi$ into finitely many $\mL_{\Div}^{\leqslant}$-definable sets $X_1,\ldots, X_n$ such that either $X_i\cap K$ is empty or $f^\sigma_{|X_i\cap K}$ has constant sign. 
\end{lemma} 

\begin{proof} Let $c\in K$ be such that $v(f^\sigma(x))=v(c)$ for all $x\in \cU_\varphi$ and $P(\xi)\in \widetilde{K}[\xi]$ be non-zero and such that $P(\res(c^{-1}f^\sigma(x)))=0$ for all $x\in \cU_\varphi$. Let $a_1,\ldots, a_n\in \widetilde{K}_\alg$ be all the distinct roots of $P$ and set $X_i:=\{x\in U_\varphi \mid \res(c^{-1}f^\sigma(x))=a_i\}$. By Theorem \ref{thm:Cmin}, $X_i$ is a finite boolean combination of balls and by Lemma \ref{lem:traces}, $X_i\cap K$ is $\mL_{\Div}$-definable. Note that if $a_i\notin \widetilde{K}$, then $X_i$ is empty. Moreover, as $v(f^\sigma(x))=v(c)$ for all $x\in \cU_\varphi$, no $a_i$ is $0$. Now, if $X_i\cap K$ is non-empty, then $f^\sigma$ cannot change sign on $X_i\cap \cU_\varphi$, as the sign of $\res^{-1}(a_i)\cap K$ is constant for every non-zero $a_i\in \widetilde{K}$. 
\end{proof}

\begin{theorem}\label{thm:weak-o-min_an}
The theory $\Th(K,\mathcal{L}_{\cA}^{\leqslant})$ is weakly o-minimal. 
\end{theorem}

\begin{proof}
It suffices to show that every elementary extension of $(K,\mathcal{L}_{\cA}^{\leqslant})$ is weakly o-minimal. Let $L$ be any such extension and $Z\subseteq L$ be a definable set. Let $x$ be a variable, $\varphi(x,y)$ be an $\mathcal{L}_{\cA}^{\leqslant}$-formula such that $Z$ is defined by $\varphi(x,a)$ for $a\in L^{|y|}$. We need to show that $Z$ is a finite union of convex sets and points. Without loss of generality we may suppose that $x$ varies over $L^\circ$ and that $a\in (L^\circ)^{|y|}$. Each term in $t(x,a)$ in $\varphi$ defines the same function as an $\mathcal{L}_{\mathcal{A}(L)}$-term $t_a'(x)$. Let $\varphi'(x)$ be the corresponding $\mathcal{L}_{\mathcal{A}(L)}^{\leqslant}$-formula where every term $t(x,a)$ in $\varphi(x,a)$ is replaced by $t_a'(x)$. Both $\varphi$ and $\varphi'$ define $Z$. By Theorem \ref{thm:QE_an} applied to $\mathrm{Th}(L,\mathcal{L}_{\mathcal{A}(L)}^{\leqslant})$, we may assume that $\varphi'(x)$ is an atomic $\mathcal{L}_{\mathcal{A}(L)}^{\leqslant}$-formula. We are reduced then to the following three cases for $t(x),s(x)$ two $\mathcal{L}_{\mathcal{A}(L)}$-terms
\begin{enumerate}
\item $\varphi'$ is of the form $t(x)=0$;
\item $\varphi'$ is of the form $0< t(x)$;
\item $\varphi'$ is of the form $v(t(x))\leqslant v(s(x))$. 
\end{enumerate} 

By Theorem \ref{thm:terms}, we may further suppose that, after possibly partitioning $L^\circ$ into a definable finite partition of $L$-annuli $\cU_i$, there is a finite set $S\subseteq L^\circ$ such that on each annulus $\cU_i$ we have that:  
\[
t(x)\vert_{\mathcal{U}_i \setminus S}=(R_i(x) \cdot E_i^\sigma) \vert_{\mathcal{U}_i \setminus S}(x) \text{ and }  s(x)\vert_{\mathcal{U}_i \setminus S}=(R_i'(x) \cdot F_i^\sigma) \vert_{\mathcal{U}_i \setminus S}(x), 
 \]
where $R_i, R_i'$ are rational functions over $L$ and $E_i^{\sigma}, F_i^{\sigma}\in \mathcal{O}_L^{\sigma}(\mathcal{U}_i)$ are strong units. Moreover, by Lemma \ref{lem:sign_change}, we may refine such a partition into a definable partition $\mathcal{P}$ such that for each $X\in \mathcal{P}$, if $X\subseteq \cU_i$, then neither $E_i^\sigma$ nor $F_i^{\sigma}$ restricted to $X$ change sign. Fix some $X\in \mathcal{P}$ such that $X\subseteq \cU_i$. We split in cases: 

$(1)$ Suppose that $\varphi'$ is of the form $t(x)=0$. Let $c\in L$ be such that $v(E_i^{\sigma})=v(c)$ for all $x\in \cU_i$. Then, the formula $\varphi'$ restricted to $X\setminus S$, is equivalent to $R_i(x)=0$, which defines a finite set of points. 

$(2)$ Suppose that $\varphi'$ is of the form $0<t(x)$. By construction the sign of $E_i^{\sigma}(x)$ does not change on $X$. Therefore, restricted to $X\setminus S$, $\varphi'$ is equivalent either to $0<R_i(x)$ or $R_i(x)<0$. The result follows then from Corollary \ref{cor:RCVF_weak}.

$(3)$ Suppose that $\varphi'$ is of the form $v(t(x))\leqslant v(s(x))$. Since both $E_i^\sigma$ and $F_i^\sigma$ are strong units, there are $b,c\in L$ such that for all $x\in \cU_i$, $v(E_i^\sigma(x))=v(c)$ and $v(F_i^\sigma(x))=v(b)$. Then, we have that $\varphi$, restricted to $X\setminus S$, is equivalent to $v(c)v(R_i(x))\leqslant v(b)v(R_i'(x))$, and again the result follows from Corollary \ref{cor:RCVF_weak}. 
\end{proof}

\section{Overconvergent real analytic structure}\label{sec:over}

Let $K$ be a rank one complete real closed valued field. Examples include the field of Puiseux series $K_0:=\bigcup_{n>0}\R(\!(t^{1/n})\!)$ and Hahn fields such as $\R(\!(t^{\Q})\!)$ or $\R(\!(t^\R)\!)$. For $x=(x_1,\ldots,x_n)$ let $K\langle x\rangle$ be the ring of converging power series over $K$, that is, those power series $\sum_{i} a_ix^i$ such that $\lim_{|i|\to\infty} v(a_i) =0$. We use $\|f\|$ to denote the Gauss norm on $K\langle x\rangle$. Let $K\llangle x\rrangle$ denote the ring of \emph{overconvergent} power series over $K$, which corresponds to the subring of $K\langle x\rangle$ consisting of those power series that converge on some ball of valuation radius strictly bigger than 1 centered at $\mathbf{0}\in K^n$. We let $\mL_{\ov(K)}$, the language of \emph{overconvergent $K$-analytic structure}, be the language $\mathcal{L}_{\Div}$ together with symbols for $^{-1}$ and for every element $f\in K\llangle x\rrangle$ for all finite tuples of variables $x=(x_1,\ldots,x_n)$. The structure $(K,\mathcal{L}_{\ov(K)})$ is defined by interpreting $\mathcal{L}_{\Div}$ and $^{-1}$ in the natural way and each $f\in K\llangle x\rrangle$ with $x=(x_1,\ldots,x_n)$ as the function
\[
f^{\mathcal{L}_{\ov(K)}}(x):=
\begin{cases}
f(x) & \text{ if $x\in (K^\circ)^n$} \\ 
0 & \text{otherwise.}
\end{cases}
\]

As before, we let $\mL_{\ov(K)}^{\leqslant}$ to be the extension of $\mL_{\ov(K)}$ by the binary relation $\leqslant$. 

\begin{remark}\label{rem:over_sep} There is a separated analytic structure $\mathcal{A}$ on $K$ such that every $\mL_{\ov(K)}^{\leqslant}$-definable set $X\subseteq K^m$ is $\mL_{\cA}^{\leqslant}$-definable. See in particular the separated overconvergent structure provided in \cite[Section 4.4, Example 10]{CLip}. 
\end{remark} 

\begin{corollary}\label{cor:overwo} The structure $(K,  \mL_{\ov(K)}^{\leqslant})$ is weakly o-minimal. 
\end{corollary} 

\begin{proof} This follows from Remark \ref{rem:over_sep} and Theorem \ref{thm:weak-o-min_an}.  
\end{proof}

Even though we know that the structure $(K,  \mL_{\ov(K)}^{\leqslant})$ is weakly o-minimal, the proof given above passes through a proper expansion of $\mL_{\ov(K)}$ to show this. It is thus natural to ask whether $(K,\mL_{\ov(K)})$ has quantifier elimination. We show this is the case, following with minor modification the argument for quantifier elimination in \cite{lip-rob-05}. We will need the following additional lemma. 

\begin{lemma}\label{lem:positive} Let $f\in K\llangle x_1,\ldots,x_n\rrangle$ be such that $\Vert f-1\Vert<\Vert f\Vert =1$. Then $0<f(x)$ for all $x\in (K^{\circ})^n$. 
\end{lemma}

\begin{proof}
Let $f(x)=\sum_{\nu\in \N^{n}} a_\nu x^\nu$. Let $\mathbf{0}$ denote the constant zero sequence in $\mathbb{N}^n$. Since $\Vert f-1\Vert<\Vert f\Vert$, we have both that $v(a_\mathbf{0}-1)<v(a_\mathbf{0})=1$ and that for every $i\neq \mathbf{0}$, $v(a_i)<1$. In particular, $a_\mathbf{0}-1\in K^{\circ\circ}$ and therefore $0<(a_0-1)+1=a_0$.  A similar argument applies for every $x\in (K^{\circ})^{n}$. Indeed, we have that $v(f(x)-a_0)<1$. Hence $f(x)-a_0\in K^{\circ\circ}$ which shows that $0<(f(x)-a_0)+a_0=f(x)$.  
\end{proof}

%
%

\begin{theorem}\label{thm:QE_ov} The theory $Th(K,\mL_{\ov(K)}^{\leqslant})$ has quantifier elimination. 
\end{theorem}

\begin{proof} The proof follows almost word for word the strategy in \cite[Theorem 2.1]{LipRob}. Let $\varphi(x,y)$ be a quantifier free $\mL_{\ov(K)}^{\leqslant}$-formula with $x=(x_1,\ldots,x_m)$ and $y=(y_1,\ldots,y_n)$. We will show that there is a quantifier free $\mL_{\ov(K)}^{\leqslant}$-formula $\psi(x)$ which is equivalent to $\exists y \varphi(x,y)$. By standard reductions and the steps used in \cite[Theorem 2.1]{LipRob}\footnote{Remark that in \cite[Theorem 2.1]{LipRob}, $t$ can be replaced by any element $c\in K$ such that $0<v(c)<1$.}, one is left with the following situation: the formula $\varphi$ is of the form $\varphi_0(x)\wedge \varphi_1(x,y)$, where $\varphi_1(x,y)$ is of the form
\[
\bigwedge_{i\in I_1} (f_i(x,y)=0) \wedge \bigwedge_{i\in I_2} (0<f_i(x,y))\wedge \bigwedge_{i\in I_3} (v(f_i(x,y))=1) \wedge \bigwedge_{i\in I_4, j\in I_5} v(f_i(x,y))<v(f_j(x,y))
\]
where 
all variables range over the valuation ring, each $f_i$ is of the form 
\[
f_i = a_{i} (x) E_i(x,y)F_i(x,y),
\]
with $a_{i}(x)\in K\llangle x\rrangle$, $F_i(x,y) \in K\llangle x,y_1,\ldots,y_{n-1}\rrangle[y_n]$ and $E_i(x,y) K\llangle x,y\rrangle$ is a unit such that $\Vert E_i -1\Vert< \Vert E_i\Vert $. By Lemma \ref{lem:positive}, $E(x,y)>0$ for all $(x,y)\in (K^{\circ})^{m+n}$. Hence, we have that for each $i, j\in \bigcup_{s=1}^5 I_s$
\begin{align*}
0 \ \square  \ f_i(x,y) & \Leftrightarrow 0 \ \square \ a_{i\mu_i}(x) F_i(x,y), \\
v(f_i(x,y)) = 1 & \Leftrightarrow v(a_{i\mu_i}(x) F_i(x,y)) = 1, \\
v(f_i(x,y)) < v(f_j(x,y)) & \Leftrightarrow v(a_{i\mu_i}(x) F_i(x,y)) < v(a_{j\mu}(x) F_j(x,y)), 
\end{align*}
where  $\square \in \{=,<\}$. Hence, by Theorem \ref{thm:QE}, we can eliminate the quantifier $\exists y_n$. The result follows by induction on $n$.  
\end{proof}

\subsection*{Acknowledgements} We would like to thank Raf Cluckers for encouraging discussions around the topics of the article.


\begin{thebibliography}{10}

\bibitem{CD}
Gregory Cherlin and Max~A. Dickmann.
\newblock Real closed rings ii. model theory.
\newblock {\em Annals of Pure and Applied Logic}, 25(3):213 -- 231, 1983.

\bibitem{clu-2003}
R.~Cluckers.
\newblock Analytic $p$-adic cell decomposition and integrals.
\newblock {\em Transactions of the American Mathematical Society},
  356(4):1489--1499, 2004.

\bibitem{CLip}
R.~Cluckers and L.~Lipshitz.
\newblock Fields with analytic structure.
\newblock {\em J. Eur. Math. Soc. (JEMS)}, 13:1147--1223, 2011.
\newblock math.LO/0610666.

\bibitem{clr-06}
R.~Cluckers, L.~Lipshitz, and Z.~Robinson.
\newblock Analytic cell decomposition and analytic motivic integration.
\newblock {\em Ann. Sci. {\'E}cole Norm. Sup. (4)}, 39(4):535--568, 2006.

\bibitem{CLR}
R.~Cluckers, L.~Lipshitz, and Z.~Robinson.
\newblock Real closed fields with non-standard and standard analytic structure.
\newblock {\em J. Lond. Math. Soc. (2)}, 78(1):198--212, 2008.

\bibitem{DD}
J.~Denef and Lou van~den Dries.
\newblock {$p$}-adic and real subanalytic sets.
\newblock {\em Ann. of Math. (2)}, 128(1):79--138, 1988.

\bibitem{L}
L.~Lipshitz.
\newblock Rigid subanalytic sets.
\newblock {\em American Journal of Mathematics}, 115(1):77--108, Feb. 1993.

\bibitem{lip-rob-98}
L.~Lipshitz and Z.~Robinson.
\newblock One-dimensional fibers of rigid subanalytic sets.
\newblock {\em The Journal of Symbolic Logic}, 63:83--88, March 1998.

\bibitem{lip-rob-05}
L.~Lipshitz and Z.~Robinson.
\newblock Uniform properties of rigid subanalytic sets.
\newblock {\em Trans. Amer. Math. Soc.}, 357(11):4349--4377 (electronic), 2005.

\bibitem{LipRob}
L.~Lipshitz and Z.~Robinson.
\newblock Overconvergent real closed quantifier elimination.
\newblock {\em Bull. London Math. Soc.}, 38(6):897--906, 2006.

\bibitem{Lip-Rob-2000}
Leonard Lipshitz and Zachary Robinson.
\newblock Rings of separated power series and quasi-affinoid geometry.
\newblock {\em Ast\'{e}risque}, (264):vi+171, 2000.

\bibitem{schoutens}
Hans Schoutens.
\newblock Rigid subanalytic sets.
\newblock {\em Compositio Math.}, 94(3):269--295, 1994.

\bibitem{vdDHM}
Lou van~den Dries, D.~Haskell, and D.~Macpherson.
\newblock One-dimensional $p$-adic subanalytic sets.
\newblock {\em Journal of the London Mathematical Society}, 59(1):1--20, 1999.

\bibitem{lewen-vdD}
Lou van~den Dries and Adam~H. Lewenberg.
\newblock {$T$}-convexity and tame extensions.
\newblock {\em J. Symbolic Logic}, 60(1):74--102, 1995.

\bibitem{Mac-Mar-Dries1}
Lou van~den Dries, Angus Macintyre, and David Marker.
\newblock The elementary theory of restricted analytic fields with
  exponentiation.
\newblock {\em Ann. of Math. (2)}, 140(1):183--205, 1994.

\bibitem{Mac-Mar-Dries3}
Lou van~den Dries, Angus Macintyre, and David Marker.
\newblock Logarithmic-exponential power series.
\newblock {\em J. London Math. Soc. (2)}, 56(3):417--434, 1997.

\bibitem{Mac-Mar-Dries2}
Lou van~den Dries, Angus Macintyre, and David Marker.
\newblock Logarithmic-exponential series.
\newblock In {\em Proceedings of the {I}nternational {C}onference ``{A}nalyse
  \& {L}ogique'' ({M}ons, 1997)}, volume 111, pages 61--113, 2001.

\end{thebibliography}
\end{document}